\documentclass{article}

\usepackage[italian, english]{babel}

\usepackage{geometry}
\usepackage[shortlabels]{enumitem}
\usepackage{mathtools, amsthm, amsfonts, amssymb}
\usepackage{graphicx, xcolor}
\usepackage[colorlinks=true, allcolors=blue]{hyperref}
\usepackage{dsfont}
\numberwithin{equation}{section}

\allowdisplaybreaks

\newtheorem{thm}{Theorem}[section]

\newtheorem{lemma}[thm]{Lemma}
\newtheorem{rmk}[thm]{Remark}
\newtheorem{pro}[thm]{Proposition}

\theoremstyle{remark}

\theoremstyle{definition}
\newtheorem{defn}[thm]{Definition}

\newcommand{\eps}{\varepsilon}

\newcommand{\gammalim}{\Gamma{\text -}\mkern-1mu\lim}

\newcommand{\abs}{|}
\newcommand\blfootnote[1]{%
  \begingroup
  \renewcommand\thefootnote{}\footnote{#1}%
  \addtocounter{footnote}{-1}%
  \endgroup
}

\newcommand\restr[2]{{
		\left.\kern-\nulldelimiterspace 
		#1 
		\vphantom{\Big|} 
		\right|_{#2} 
}}

\title{
\Large A note on the homogenization of incommensurate thin films}
\date{}
\author{Irene Anello, Andrea Braides and Fabrizio Caragiulo
\\ 
SISSA, via Bonomea 265, Trieste, Italy}

\begin{document}
\maketitle
\blfootnote{Preprint SISSA 22/2022/MATE}

\begin{abstract}
Dimension-reduction homogenization results for thin films have been obtained under hypotheses of periodicity or almost-periodicity of the energies in the directions of the mid-plane of the film. In this note we consider thin films, obtained as sections of a periodic medium with a mid-plane that may be incommensurate; that is, not containing periods other than $0$. A geometric almost-periodicity argument similar to the cut-and-project argument used for quasicrystals allows to prove a general homogenization result.
\end{abstract}

\def\e{\varepsilon}

\section{Introduction}
The energy of heterogeneous thin films of a hyperelastic material in $\mathbb{R}^{d+1}$ can be described by integral functionals of the form
\begin{equation}\label{heterog-e}
\frac{1}{2\eps}\int_{\omega\times(-\eps,\eps)} f_\e(x, \nabla u )\,\mathrm{d}x, 
\end{equation}
where $f_\e$ are hyperelastic energy densities satisfying suitable growth conditions, $\omega\subset \mathbb{R}^{d}$ is the middle section of the thin film, $2\e$ is its thickness, and $u:\omega\times(-\eps,\eps)\to \mathbb R^m$ is the  displacement (in the physical case, $d=2$ and $m=3$). We use $\mathrm{d}x$ to denote integration with respect to the Lebesgue measure in $\mathbb{R}^{k}$ both for $k=d$ and $k={d+1}$, which one being clear from the context. 
In the seminal paper by Le Dret and Raoult \cite{LDR}, in the case  $f_\e=f$ independent of $\e$ and homogeneous (i.e., $x$-independent), the behaviour of such energies as $\e\to 0$ has been proven to be described by a dimensionally reduced hyperelastic energy of the form
\begin{equation}\label{homog-re}
\int_{\omega} f^{\rm LDR}(\nabla v)\,\mathrm{d}y,
\end{equation}
where now $v:\omega\to \mathbb R^m$  has a domain identified as the mid-section of the film, and $f^{\rm LDR}$ is a quasiconvex function explicitly described from $f$.
For general energies \eqref{heterog-e} a compactness theorem \cite{bff} ensures that, up to subsequences, their behaviour can be described by a possibly heterogeneous dimensionally reduced energy
\begin{equation}\label{homog-hom}
\int_{\omega} f_0(y,\nabla v)\,\mathrm{d}y,
\end{equation}
with $f_0$ possibly depending on the subsequence. 
In the case of a periodic integrand we can suppose, as this is the relevant case, that the oscillation be at the same scale of the thickness, so that we assume $f_\e(x,A)= f\bigl({x\over\e},A\bigr)$ with $f(\cdot,A)$ periodic in the coordinate directions. In this case we have {\em homogenization} \cite{bff}; that is, the limit $f_0$ is homogeneous and independent of the subsequence of $\e$, and described by a suitable asymptotic formula. In order to prove this result, a key argument is the invariance property of the energies by translations of the form $\e e_j$ for $j\in\{1,\ldots, d\}$, where $\{  e_i: i\in \{1,\ldots, d+1\}\}$ is the canonical orthonormal basis of $\mathbb{R}^{d+1}$, entailing the translation invariance of $f_0$. This argument is made possible by the assumption that the middle plane of the thin film contains a $d$-dimensional lattice of periods for $f$; that is, it is `commensurate' with $\mathbb{Z}^{d+1}$. 

In this paper we consider the general case of a periodic energy density $f$, when the middle plane of the thin film contains a $n$-dimensional lattice of periods for $f$ with $n\le d$. `Incommensurate' thin films are those with $n=0$; that is, when the middle plane of the thin film does not contain any period for $f$ except $0$. 

In order to treat general, possibly incommensurate, thin films, we introduce some notation, slightly different from the one used for commensurate thin films, due to the necessity to distinguish between the periods of the energy density and the directions of the thin film. We choose to maintain $\mathbb{Z}^{d+1}$ as set of periods, considering a Carath\'eodory function 
$\widetilde f:\mathbb{R}^{d+1} \times \mathbb{M}^{m\times{(d+1)}} \to \left[0,+\infty\right)$  satisfying the \emph{standard p-growth conditions} 
\begin{equation}\label{eq:GrowthCondition}
    \alpha|A|^p \leq\widetilde f(x,A) \leq \beta(1+|A|^p)
\end{equation}
for all $x \in \mathbb{R}^{d+1}$ and $A \in \mathbb{M}^{m \times {d+1}}$ and some $\alpha,\beta>0$,
which is $1$-periodic in all coordinate directions; i.e.,
\begin{equation}\label{eq:periodicity}
  \widetilde f(x+ {e}_i,A) = \widetilde f(x,A)
\end{equation}
for all $x\in \mathbb{R}^{d+1}$, all matrices $A\in \mathbb{M}^{m\times{(d+1)}}$ and all vectors $ {e}_i$ of the standard orthonormal basis of $\mathbb{R}^{d+1}$.  We use the notation $\mathbb{M}^{m\times k}$ for the space of ${m\times k}$ matrices with real entries.

We fix a unit vector $\nu$ and consider the hyperplane $\Pi=\{x: \langle x,\nu\rangle=0\}$.
The interesting case, to keep in mind as the most relevant one, is when $\Pi$ is an \emph{irrational} hyperplane in $\mathbb{R}^{d+1}$; that is, such that
\begin{equation}\Pi \cap \mathbb{Z}^{d+1} = \{0\}. \end{equation}
We fix $\widetilde{\omega}$ a bounded subset of $\Pi$, open in the relative topology, $h>0$, and for each $\e>0$ consider the functional
\begin{equation}\label{eq:funcdef}
    \widetilde{I}_{\eps}(u) = \frac{1}{2h\eps}\int_{\widetilde{\omega}+(-h\eps,h\eps)\nu}\widetilde f\Big(\frac{x}{\eps}, \nabla
    {u}\Big)\,
    \mathrm{d}x,
\end{equation}
with domain $\mathrm{W}^{1,p}(\widetilde{\omega}+(-h\eps,h\eps)\nu,\mathbb{R}^m)$,
which ideally represents the elastic free energy of a thin film of size $2h\eps>0$ around $\widetilde{\omega}$. 
The introduction of a constant $h>0$ amounts to assuming that thickness and periods are comparable, so that $h$ represents their ratio, which we highlight for possible future reference. 
We will prove that there exists a function $\widetilde f_{\text{hom}} \colon \mathbb{M}^{m\times d}\to [0,+\infty)$ 
satisfying an asymptotic formula, such that
$$
\gammalim_{\eps\to 0}\widetilde{I}_\eps(u)=\int_{\widetilde{\omega}}\widetilde f_{\text{hom}}\big(\nabla u(y)\big)\,\mathrm d y.$$
The $\Gamma$-limit is performed in a dimension-reduction fashion, using a convergence of $u_\e\in \mathrm{W}^{1,p}(\widetilde{\omega}+(-h\eps,h\eps)\nu,\mathbb{R}^m)$ to $u\in  \mathrm{W}^{1,p}(\widetilde{\omega},\mathbb{R}^m)$, where $\widetilde{\omega}$ is identified with a subset of $\mathbb{R}^d$ in order for the integration to be well defined.

The result is obtained by first resorting to the theory of \cite{bff}. To that end we rewrite the functionals in the form \eqref{heterog-e}, with $f_\e(x,A)= f\smash{\bigl({x\over\e},A\bigr)}$, and $f$ obtained from $\smash{ \widetilde f}$ by a linear change of variables. 
Using the dimension-reduction convergence of functions as in \cite{LDR}, we can then apply the compactness theorem in \cite{bff} to obtain a limit of the form \eqref{homog-hom}. Since $f$ may not be periodic in the coordinate directions we cannot immediately conclude that homogenization takes place, since an invariance-by-translation argument does not apply. However, we can use a sort of {\em geometric almost periodicity} property: the set of periods for $\widetilde f$ which are close to $\Pi$ (closeness suitably quantified by a small parameter) is projected to a uniformly dense set in $\Pi$. We remark that the former set corresponds also to the set of periods for $ f$ which are close to the hyperplane identified with $\mathbb{R}^d$, This argument reminds the {\em cut-and-project} arguments typical of quasicrystalline structures (see \cite{deB,Bou,Fon,WGC,BS,BCS}). 
The existence of such geometric quasi-periods is not sufficient to prove the necessary translation-invariance properties for $f_0$ and the homogenization asymptotic formula. To that end it is necessary to construct test functions by using translation arguments, which are not directly at hand since translation by a quasi-period may exit the thin film domain. The key technical point of the paper is a novel lemma, which ensures that in the constructions of test functions it is sufficient to modify functions defined in a smaller thin film, which is then mapped inside the original thin film by any of the above-mentioned translations (scaled by $\e$).

\subsection{Statement of the results}
We now formalize what we anticipated in the Introduction, rewriting the energies $\widetilde{I}_{\eps}$ in \eqref{eq:funcdef} in order to apply the results in \cite{bff} with more ease. To that end, we make a change to coordinates more suitable to the problem: we let $\phi$ be the linear isometry in $\mathbb{R}^{d+1}$ sending $ {e}_1,\ldots,  {e}_d$ to an orthonormal basis $ {\pi}_1, \ldots,  {\pi}_d$ of $\Pi$ and $ {e}_{d+1}$ to $\nu$. Let $\omega$ be the open set in $\mathbb R^d$ such that $\phi^{-1}(\widetilde{\omega})=\omega\times \{0\}$, so that
\[\phi^{-1}\big(\widetilde{\omega}+(-h\eps,h\eps)\nu\big)=\omega\times (-h\eps,h\eps)\eqqcolon \Omega_\eps\subset \mathbb{R}^d\times \mathbb{R}\]
After this change of variables, setting
\begin{equation}\label{eq:funcdef2}
    {I}_{\eps}(u) = \frac{1}{2h\eps}\int_{\Omega_{\eps}}f\Big(\frac{x}{\eps}, \nabla{u(x)}\Big)\,\mathrm{d}x,
\end{equation}
where $u \in \mathrm{W}^{1,p}(\Omega_\eps,\mathbb{R}^m)$ and, denoted by $R$ the constant matrix equal to $\nabla\phi$, having set
\begin{equation}\label{eq:Integrand}
f(x,A)=\widetilde f(\phi(x), AR ),
\end{equation}
we have 
\begin{equation}\label{eq:equiv}
{I}_{\eps}(u)= \widetilde {I}_{\eps}(\widetilde u),\quad\hbox{  where }\quad\widetilde u(\widetilde x)= u(\phi^{-1}(\widetilde x)).
\end{equation}
Hence, the functionals ${I}_{\eps}$ and $\widetilde {I}_{\eps}$ are equivalent, up to a linear change of variables.
In order to simplify the statement of the convergence, as is customary we scale all functionals to a common domain, 
obtaining
\begin{equation}\label{eq:funcdef3}\begin{split}
    F_\eps(u)&\coloneqq\frac{1}{2h}\int_{\omega\times (-h,h)} f\Big({x\over\e},y, \Big(\nabla_x u(x,y)\Big| {1\over\e}{\partial_y u(x,y)}\Big)\Big)\,\mathrm{d}x\,\mathrm{d}y 
\end{split}\end{equation}
for $u\in W^{1,p}(\omega\times (-h,h),\mathbb R^m)$, so that $F_\eps(\overline u)= I_\e(u)$, where $\overline u(x,y)=u(x,\e y)$.
Here, we have rewritten, with a little abuse of notation, a variable in $\mathbb R^{d+1}$ as a pair $(x,y)$ with $x\in \mathbb R^d$,
and identified a matrix $A\in \mathbb{M}^{m\times {d+1}}$ with a pair which we denote $(A'|\xi)$, where $A'\in \mathbb{M}^{m\times d}$ and $\xi\in \mathbb R^m$ is a (column) vector. The notation $\nabla_xu$ denotes the gradient with respect to the coordinates of $x$. We will keep the standard notation $\nabla u$ when the gradient is performed with respect to all the variables in the domain, be it in $\mathbb R^d$ or $\mathbb R^{d+1}$.
 
\goodbreak
We will use the following notion of convergence.
\begin{defn}[convergence to dimensionally reduced parameters]\label{def:convfun}
A sequence of functions $u_\e$ with $u_\e\in {W}^{1,p}(\Omega_\eps,\mathbb{R}^m)$ {\em converges to} $u\in {W}^{1,p}(\omega,\mathbb{R}^m)$ if the corresponding functions $\overline u_\e$ defined by $\overline u_\e(x,y)=u_\e(x,\e y)$ converge to some $\overline u$ weakly in $W^{1,p}(\omega\times (-h,h),\mathbb R^m)$ and $\overline u(x,y)=u(x)$. 
\end{defn}

We recall that $I_\e$ are equicoercive with respect to this convergence, in the sense that if $I_\e(u_\e)$ is equibounded and $\overline u_\e$ are bounded in $L^p(\Omega\times (-h,h),\mathbb R^m)$ then $u_\e$ converge to some $u$, up to subsequences \cite{bff}.

\begin{defn}[convergence to dimensionally reduced energies]\label{def:convergence_functionals} We say that $I_\e$ $\Gamma$-{\em converges to} $F_0$ with respect to the convergence of $u_\e$ above if the corresponding $F_\e$ $\Gamma$-converge to the functional $\overline F_0$ on functions independent of $y$, defined by $F_0(\overline u)= \overline{F}_0(u)$ if $\overline u(x,y)=u(x)$, with respect to the weak convergence in $W^{1,p}$.
\end{defn}

We can then state the homogenization result as follows.

\begin{thm}[homogenization theorem]\label{thm:MainResult}
Let $\nu$ be a unit vector and $\Pi=\{x: \langle x,\nu\rangle=0\}$ be a hyperplane in $\mathbb R^{d+1}$.
Let $\smash{\widetilde f\colon \mathbb{R}^{d+1}\times \mathbb{M}^{m\times (d+1)}\to \mathbb R}$ be a Carathéodory function satisfying the $p$-growth conditions \eqref{eq:GrowthCondition} for $p>1$ and periodic in the coordinate direction as in \eqref{eq:periodicity}. Let $\phi$ be the linear isometry defined above and let $f\colon \mathbb{R}^{d+1}\times \mathbb{M}^{m\times (d+1)}\to \mathbb R$ be the Carath\'eodory integrand defined by \eqref{eq:Integrand}.
 Let $h>0$, let $\omega$ be an open and bounded subset of $\mathbb R^d$ and let $I_\eps$ be defined by \eqref{eq:funcdef2}, where $\Omega_\e=\omega\times (-h\e,h\e)$. Then $I_\e$ $\Gamma$-converges to $F_{\rm hom}$ with respect to the convergence in Definition {\rm\ref{def:convfun}} as $\e\to 0$, where 
\begin{equation}\label{eq:ffhom}
F_{\rm hom}(u)=\int_\omega f_{\rm hom}(\nabla u)\,\mathrm dx
\end{equation}
for $u\in W^{1,p}(\omega,\mathbb R^m)$, and $f_{\rm hom}\colon \mathbb{M}^{m\times d}\to \mathbb R$ is a quasiconvex function satisfying a $p$-growth conditions and the {\em asymptotic homogenization formula}
        \begin{equation}\label{eq:HomogenizationFormula}
        \begin{split}
            f_{\mathrm{hom}}(A) = \lim_{T \to +\infty}\frac1{T^d} \inf_{u\in \mathcal W_T} \bigg\{ \frac{1}{2h}\int_{(0,T)^d\times(-h,h)}&f\big(x,y,\big(A+\nabla_x u\big| {\partial_y u}\big)\big) \,\mathrm dx\,\mathrm dy \Bigg\}
        \end{split}
        \end{equation}
         where $$\mathcal{W}_T\coloneqq \Big\{
                 u \in \mathrm W^{1,p}\big((0,T)^d\times(-h,h),\,\mathbb{R}^m\big) \colon
                 u=0 \text{ on } \partial\big((0,T)^d\big)\times(-h,h)\Big\}.$$
\end{thm}

\begin{rmk}[homogenization on $\Pi$]\rm
By using \eqref{eq:equiv} we can interpret the result as a homogenization theorem directly on $\Pi$, with the related homogenized energy
$$
\widetilde F_{\rm hom}(u)=\int_{\widetilde \omega} f_{\rm hom}(\nabla u \, \widetilde R^{-1})\,\mathrm d{\mathcal H}^d(x),
$$
where $\widetilde R$ is the matrix related to the restriction of the linear isometry to the subspace of $\mathbb R^{d+1}$ parameterized as $\mathbb R^d$, and the Sobolev space $W^{1,p}(\widetilde \omega,\mathbb R^m)$  (with underlying measure the $d$-dimensional Hausdorff measure  restricted to $\widetilde \omega$) is suitably defined.
\end{rmk}

 \begin{figure}[ht!]
\centerline{\includegraphics[width=.8\textwidth]{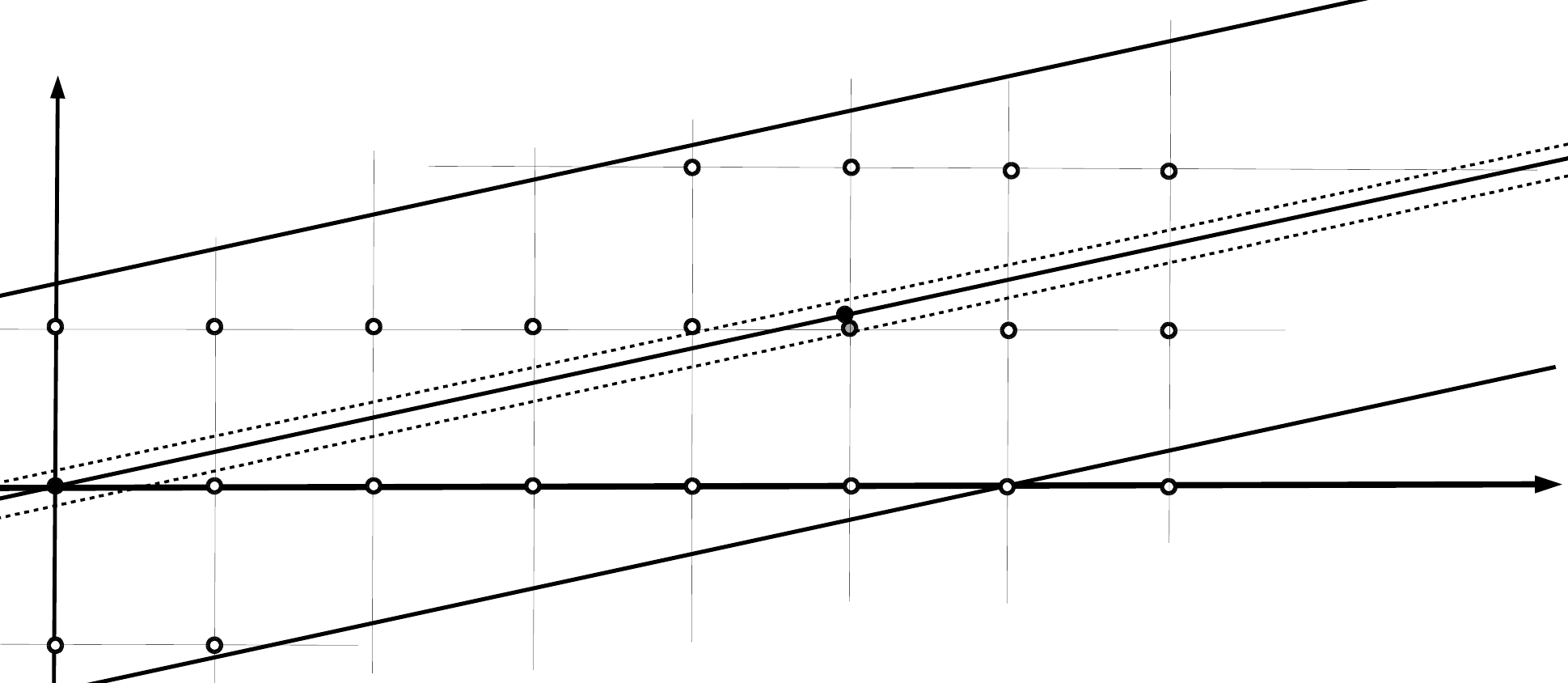}}
\caption{An irrational $\mathbb Z^2$-periodic thin film in $\mathbb R^2$ with an $\eta$-neighbourhood of its middle line.}
\label{2022ABC-Fig1}
\end{figure}
\begin{rmk}[connections with almost periodicity]\label{rem:almost-p}\rm The theorem above has been proved in \cite{bff} if $f$ is periodic in the first variable, which is the case if the lattice $\Pi \cap \mathbb{Z}^{d+1}$ has dimension $d$ (or, equivalently, if it spans  $\Pi$). 
In the case at hand  we will use a sort of {\em geometric quasi-periodicity}. Namely, we will use the fact that, for fixed $\eta>0$ the set $\mathcal T_\eta$ of $\tau\in \Pi$ such that dist$(\tau,\mathbb Z^{d+1})<\eta$ is uniformly dense in $\Pi$; that is, there exists an inclusion length $L_\eta>0$ such that $\mathcal T_\eta+[0,L_\eta]^{d+1}\supset \Pi$. This is an immediate consequence of the periodicity of the function $\psi(x)={\rm dist}(x,\mathbb Z^{d+1})$ on $\mathbb R^{d+1}$ and the consequent quasi-periodicity of its restriction to $\Pi$. This property in turn implies its uniform almost periodicity (see \cite{besic} Definition 1.7), which exactly states that for all $\eta>0$ there exist a uniformly dense set of $\eta$-{\em almost periods}; i.e., $\tau$ such that $|\psi(x+\tau)-\psi(x)|<\eta$ for all $x\in \Pi$. In particular, taking $x=0$, we have that $\psi(\tau)\le\eta$, which proves the claim. Note that this property is most relevant if $\Pi$ is irrational, and is trivial if $\Pi \cap \mathbb{Z}^{d+1}$ has dimension $d+1$. In Fig.~\ref{2022ABC-Fig1}, in a two-dimensional setting, we picture an element of ${\mathcal T}_\eta$ on $\Pi$ (black dot) and its closest element in $\mathbb Z^2$ (grey dot). 

This quasi-periodicity argument also shows that if $\widetilde f$ is continuous in the first variable uniformly with respect to the second one, then $f$ is almost periodic in the $x$-directions uniformly with respect to the second variable, and we can apply the results of \cite{braidesh} Chapter 24. This observation suggests that we can then generalize Theorem \ref{thm:MainResult} by supposing that for all $\eta>0$ there exists a uniformly dense set  $\mathcal T_\eta$ in $\Pi$ such that for all $\tau\in \mathcal T_\eta$ there exists $z_\tau\in 
\mathbb{R}^{d+1}$ such that $\|\tau-z_\tau\|<\eta$ and $|\widetilde f(x+z_\tau,A)-\widetilde f(x,A)|\le \eta(1+|A|^p)$ for all $x\in\mathbb R^{d+1}$ and $A$. This trivially holds if $\widetilde f$ is periodic as above taking $z_\tau\in\mathbb Z^{d+1}$.
\end{rmk}

The geometric quasi-periodicity property highlighted above must be complemented by a lemma, which will be used to cope with the fact that periods may not belong to $\Pi$, so that translation arguments within the thin film cannot be directly applied.

\begin{lemma}\label{Lemma:EstimateOnSlices}
Let $h>0$ and let $g:[0,h] \to [0,+\infty)$ be an integrable function and define 
\begin{equation*}
     C\coloneqq \int_0^hg(y)\,\mathrm dy  < +\infty.
\end{equation*}
Then,  for every $\delta>\eta>0$, the set 
\begin{equation} \label{eq:g(y)Inequality} 
E_\eta \coloneqq \bigg\{y\in [h-\delta,h]\,:\,  (h+\eta-y)g(y) \leq \frac{C}{\log \big(\delta/\eta\big)}\bigg\} \end{equation}
has positive measure.
\end{lemma}
\begin{proof}
Suppose, by contradiction, that for some $\delta$ there exists a $\eta$ such that  $E_{\eta}$  has null measure, then
\begin{equation*}
    (h+\eta-y)g(y) > \frac{C}{\log \big(\delta/\eta\big)},\quad \text{ for almost all } y \in [h-\delta,h].
\end{equation*}
Thus the strict inequality would persist under integration, and
\begin{eqnarray*}
    C & \ge& \int_{[h-\delta,h]}g(y)\,\mathrm dy> \frac{C}{\log \big(\delta/\eta\big)}\int_{[h-\delta,h]}\frac{1}{h+\eta-y}\,\mathrm dy\\
      & =& \frac{C}{\log \big(\delta/\eta\big)}\big(\log(\delta+\eta)-\log(\eta)\big)  
       = C\frac{\log{\big(\delta/\eta +1 \big)}}{\log{\big(\delta /\eta \big)}}>C
\end{eqnarray*}
providing a contradiction.
\end{proof}

\begin{figure}[h!]
\centerline{\includegraphics[width=.7\textwidth]{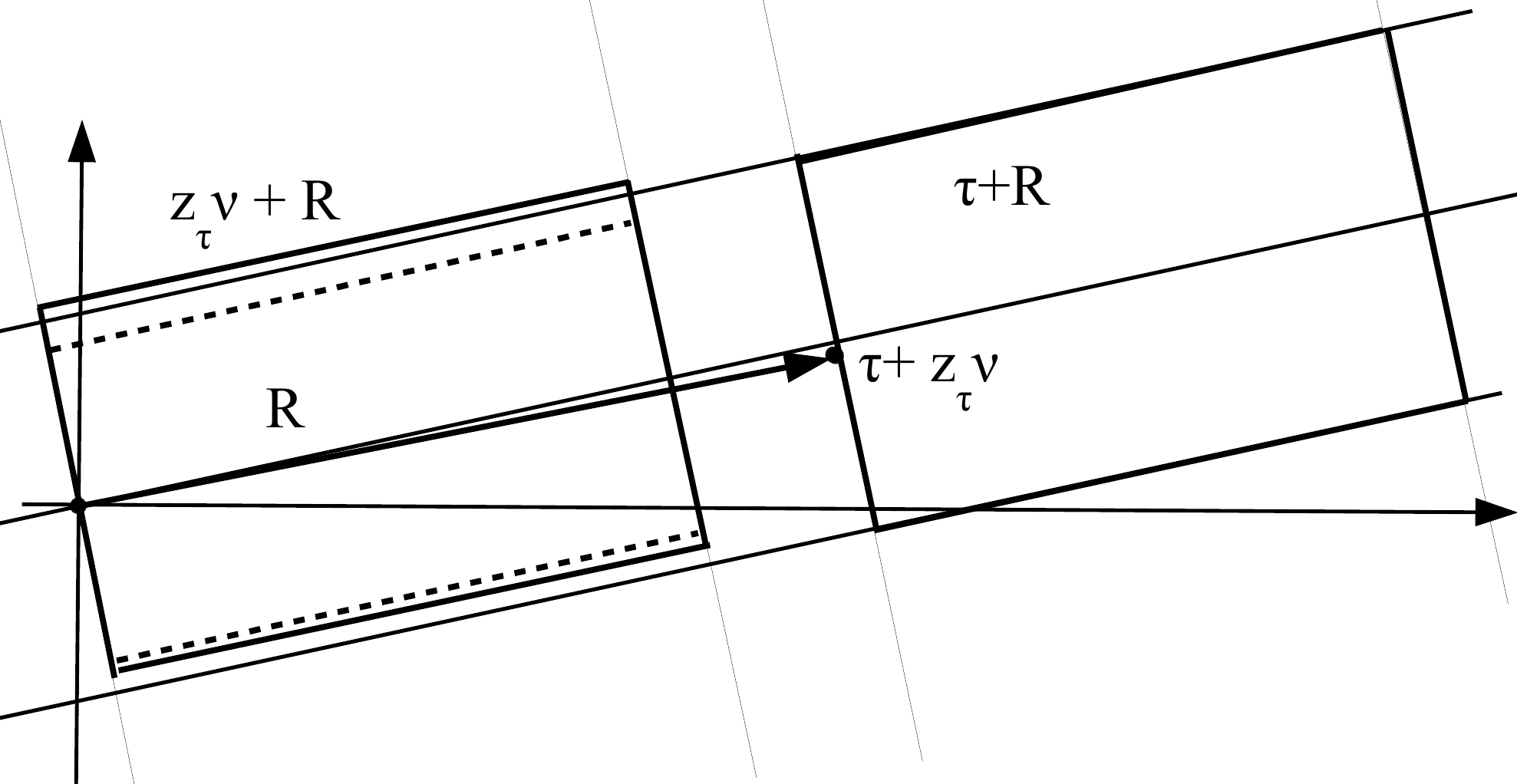}}
\caption{Translation of a rectangle $R$ by an $\eta$-almost period.}
\label{2022ABC-Fig3}
\end{figure}
The key geometric argument in the proof will be the use of Lemma \ref{Lemma:EstimateOnSlices} to define test functions on translated copied of sets by an almost period. In Fig.~\ref{2022ABC-Fig3} we have drawn a cartoon in two dimensions of the argument in the original irrational thin-film geometrical setting (that is, before changing variables with $\phi$), supposing $\tilde f$ periodic: we have a function $u_R$ defined on a rectangle $R$ inside the thin film and need to define a function in the translation of $R$ by some $\eta$-almost period $\tau\in\Pi$ with an energy that differs very little from that of the original function. We cannot directly use a translation by the period corresponding to $\tau$, which in this setting is of the form $\tau+z_\tau\nu\in\mathbb Z^2$, since to do this we would need that the function be defined on the (slightly) translated rectangle $z_\tau\nu+R$. We then restrict the original function $u_R$ slightly inside $R$ is such a way that we control the energy on the upper and lower boundary (dotted lines) and extend this restriction to a function $\tilde u_R$ defined on the whole strip orthogonal to the thin film, with controlled energy. This can be done by a construction that uses Lemma \ref{Lemma:EstimateOnSlices}. It is then possible to simply use a translation of this extended function by $\tau+z_\tau\nu$, and restrict it to $\tau+R$.

\section{Proof of the homogenization theorem}
We directly prove the theorem under the more general assumptions in Remark \ref{rem:almost-p}, which can be stated as follows for the function $f$: for all $\eta>0$ there exists $L_\eta>0$ and $\mathcal T_\eta$ in $\mathbb R^d$ such that 
$$
\mathcal T_\eta+[0,L_\eta]^d=\mathbb R^d,
$$
and for all $\tau\in \mathcal T_\eta$ there exists $z_\tau\in \mathbb{R}$ such that $|z_\tau|<\eta$ and 
\begin{equation}\label{eq:tauz}
| f(x+(\tau,z_\tau),A)-f(x,A)|\le \eta(1+|A|^p)
\end{equation} for all $x\in\mathbb R^{d+1}$ and $A$.

We first note that, by the compactness theorem for thin structures (\cite{bff} and  \cite{braidesh} Theorem 24.20) for every sequence $\{\eps_k\}_k$ of positive real numbers, $\eps_k \to 0$, there exist a subsequence (still denoted by $\eps_k$) and a Carathéodory function $f_0:\mathbb R^d \times \mathbb{M}^{m\times d} \to [0,+\infty)$ satisfying
\begin{equation*}
    0 \leq f_0(x,A) \leq \beta(1 + |A|^p),
\end{equation*}
 for all $x \in \omega$ and $A \in \mathbb{M}^{m\times d}$ such that 
\begin{equation}\label{eq:compek}
    \Gamma\hbox{-}\mkern-10mu\lim_{k \to +\infty} I_{\e_k}(u) = \int_{\omega} f_0(x,\nabla u)\,\mathrm{d}x
\end{equation}
for all $u \in W^{1,p}(\omega,\mathbb{R}^m)$ with respect to the convergence in Definition \ref{def:convfun}.
We then have to prove that $f_0$ does not depend on the space variable and is given by formula \eqref{eq:HomogenizationFormula}, which also implies that it does not depend on the subsequence $\e_k$.
To that end, we will use localization arguments and results ensuring the possibility of fixing boundary values taken from \cite{bff} and Lemma \ref{Lemma:EstimateOnSlices}.

\smallskip
In the following result we prove that $f_0$ depends only on the gradient, $f_0(x,A) = f_0(A)$. 
 
\begin{pro} 
\label{proposition:f0hom}
For any sequence $\{\eps_k\}_k$ such that the $\Gamma$-limit in \eqref{eq:compek} exists, $f_0$ depends only on the gradient; that is, $f_0(x,A) = f_0(A)$.
\end{pro}

\begin{proof}[Proof]
For any open subset $U\subset \omega$ we consider the \emph{localized} functionals
\begin{equation*}
    \begin{split}
        F_{\eps_k}(u,U)&\coloneqq  \frac{1}{2h}\int_{U\times(-h,h)} f\big({\e_k^{-1}}x,y, \big(\nabla_x u\big| {\e_k^{-1}}{\partial_y u}\big)\big)\mathrm{d}x\,\mathrm{d}y \\
        F_0(u,U) &\coloneqq {1\over2h}\int_{U\times(-h,h)} f_0(x,\nabla_xu)\,\mathrm{d}x\,\mathrm{d}y.
    \end{split}
\end{equation*}
Let $B^d_\rho(x)$ denote the ball of radius $\rho$ and center $x$ in $\mathbb R^d$.
It suffices to prove that given any $x', x'' \in \omega$ and $\rho > 0$ such that
$B^d_\rho(x')\Subset \omega$ and $B^d_\rho(x'') \Subset \omega$, and any $A \in \mathbb{R}^{m \times d}$, we have
\begin{equation}
    F_0\big(Ax, B^d_\rho(x')\big) =  F_0\big(Ax, B^d_\rho(x'')\big).
\end{equation}

Up to fixing lateral boundary values, we find a sequence $\{u_k\}_k \in \mathrm W^{1,p}\big(B^d_\rho(x') \times (-h,h),\mathbb{R}^m\big)$ with $u_k=0 \text{ on } \big(\partial B^d_\rho(x')\big)\times (-h,h)$  and such that $u_k  \to 0$ in $\mathrm L^p(\Omega_1,\mathbb{R}^m)$ and 
\begin{equation}
\lim_{k\to+\infty} F_{\eps_k}\big(Ax + u_k, B^d_\rho(x')\big) = F_0\big(Ax, B^d_\rho(x')\big)
\end{equation}
 (see \cite{bff}).

Let $\delta>0$ be fixed, and let $\eta\in(0,\delta)$. From now on we use the shorthand $\partial_y$ for the partial derivative with respect to $y$. Applying Lemma \ref{Lemma:EstimateOnSlices} to the integrable functions
\begin{equation*}
   g(t)=\int_{B^d_\rho(x')} \big|\big(A+\nabla_x u_k(x,t)\big|\eps_k^{-1}\,\partial_yu_k(x,t)\big)\big|^p\,\mathrm dx
\end{equation*}
for $t\in [0,h]$, we get the existence of $y_{\eta}^+ \in (h-\delta,h)$ such that 
\begin{eqnarray*}
     &&{\int_{y_{\eta}^+}^{h+\eta}}\int_{B^d_\rho(x')}\big|\big(A+\nabla_x u_k(x, y_{\eta}^+)\big|\,\eps_k^{-1}\partial_y u_k(x,y_{\eta}^+)\big)\big|^p\,\mathrm dx\, {\mathrm dy}\\
     \nonumber &\leq&  \frac{1}{|\log (\delta/\eta)|} \int_{B^d_\rho(x') \times (0,h)}  \big|\big(A+\nabla_x u_k(x,y_{\eta}^+)\big|\,\eps_k^{-1}\partial_y u_k(x,y_{\eta}^+)\big)\big|^p\,\mathrm dx\,\mathrm dy\\
     &\leq&  \frac{C_k}{\alpha|\log (\delta/\eta)|}, 
\end{eqnarray*}
where  $C_k\coloneqq F_{\eps_k}(Ax+u_k, B^d_\rho(x'))$.
Hence, we obtain
\begin{eqnarray}\label{eq:SliceEstimateU}
     &&\nonumber \hskip-1cm{\int_{y_{\eta}^+}^{h+\eta}}\int_{B^d_\rho(x')}f\Big (\eps_k^{-1}x,y,\big(A+\nabla_x u_k(x,y_{\eta}^+)\big|\,\eps_k^{-1}\partial_y u_k(x,y_{\eta}^+)\big)\Big) \,\mathrm dx \, {\mathrm dy} \\
     &\leq& \beta\left(|B^d_\rho(x')|{(\delta + \eta)} + \frac{C_k}{\alpha|\log (\delta/\eta)|}\right)\,,\\
      &=& \beta\left(c\rho^d{(\delta + \eta)} + \frac{C_k}{\alpha|\log (\delta/\eta)|}\right)\,,
\end{eqnarray}
being $c\coloneqq |B^1|$. Similarly, we obtain the existence of $y_{\eta}^- \in (-h,-h+\delta)$ such that
\begin{eqnarray}\label{eq:SliceEstimateU2}\nonumber
   && \hskip-1cm{\int_{-h-\eta}^{y_{\eta}^-}}\int_{B^d_\rho(x')}f\Big (\eps_k^{-1}x,y,\big(A+\nabla_x u_k(x,y_{\eta}^-)\big|\,\eps_k^{-1}\partial_y u_k(x,y_{\eta}^-)\big)\Big) \,\mathrm dx\,{\mathrm dy}\\  &\leq& \beta\left(c \rho^d {(\delta + \eta)} + \frac{C_k}{\alpha|\log (\delta/\eta)|}\right).
\end{eqnarray}
 Define $\widetilde{u}_k: B^d_\rho(x')\times \mathbb R\to\mathbb R^m$ as
\[\widetilde{u}_k(x,y)\coloneqq \begin{cases} u_k(x,y_\eta^+) &\text{ if } y\ge y^+_\eta\\
u_k(x,y) &\text{ if } y_{\eta}^-\le y \le y_\eta^+ \\
u_k(x,y_\eta^-) &\text{ if } y\le y^-_\eta.
\end{cases}\]
Note that $\widetilde{u}_k(x,y)=0$ if $x\in \partial B^d_\rho(x')$.

Let  $\{\tau_k\}_k$  be a sequence in $\e_k\mathcal{T}_{\eta}$ (the set of scaled almost-periods) such that 
$\tau_k\to x''-x'$. Such a sequence exists by the uniform density of $ \mathcal{T}_{\eta}$. Let $z_k=z_{\tau_k}$ be such that \eqref{eq:tauz} holds,
and define $v_k(x,y) = \widetilde{u}_k(x-\tau_k,y- z_k)$. Noting that $v_k$ can be extended by $0$ outside $\big(\tau_k +  B^d_\rho(x') \big)\times (-h,h)$ we get from almost-periodicity, (\ref{eq:SliceEstimateU}) and (\ref{eq:SliceEstimateU2}):
\begin{eqnarray*}
        &&F_{\eps_k}\big(Ax + v_k,\,\tau_k +  B^d_\rho(x') \big)\\
       & = &\frac1{2h}\int_{\big(\tau_k +  B^d_\rho(x')\big) \times (-h,h)} f\bigg(\frac{x}{\eps_k},y, \Big(A + \nabla_x v_k\Big| \frac{1}{\eps_k}\partial_y v_k\Big)\bigg)\mathrm{d}x\,\mathrm{d}y\\
        &= &\frac1{2h} \int_{ B^d_\rho(x') \times(-h-z_k,h+z_k)} f\bigg(\frac{x}{\eps_k}+\frac{\tau_k}{\eps_k}, y +  z_k, \Big(A +\nabla_x \widetilde{u}_k\Big| \frac{1}{\eps_k}\partial_y \widetilde{u}_k\Big)\bigg)\mathrm{d}x\,\mathrm{d}y\\
        &\le & \frac1{2h}\int_{ B^d_\rho(x') \times(-h-\eta,h+\eta)} f\bigg(\frac{x}{\eps_k}+\frac{\tau_k}{\eps_k}, y +  z_k, \Big(A +\nabla_x \widetilde{u}_k\Big| \frac{1}{\eps_k}\partial_y \widetilde{u}_k\Big)\bigg)\mathrm{d}x\,\mathrm{d}y\\
       & \leq &\frac{\beta c\rho^d(\eta+\delta)}h\\
        && \mkern+50mu + \frac1{2h}\Big(1+\frac{2\beta}{\alpha|\log (\delta/\eta)|} \Big)\int_{ B^d_\rho(x') \times(-h,h)} f\bigg(\frac{x}{\eps_k}+\frac{\tau_k}{\eps_k}, y +  z_k, \Big(A +\nabla_x {u}_k\Big| \frac{1}{\eps_k}\partial_y {u}_k\Big)\bigg)\mathrm{d}x\,\mathrm{d}y\\
       & \leq & \frac{\beta c\rho^d}h +\frac1{2h} \Big(1+\frac{2\beta}{\alpha|\log (\delta/\eta)|} \Big)\int_{B^d_\rho(x') \times (-h,h)} f\bigg(\frac{x}{\eps_k},y , \Big(A +\nabla_x u_k\Big| \frac{1}{\eps_k}\partial_y u_k\Big)\bigg)\mathrm{d}x\,\mathrm{d}y\\
        && \mkern+50mu +  \frac\eta{2h}\Big(1+\frac{2\beta}{\alpha|\log (\delta/\eta)|} \Big) \int_{B^d_\rho(x') \times (-h,h)} \Bigl( 1 + \Big\abs \Big(A +\nabla_x u_k\Big| \frac{1}{\eps_k}\partial_y u_k\Big)\Big\abs ^p\Bigr)\mathrm{d}x)\,\mathrm{d}y \\\\
        &\le & \frac{\beta c\rho^d(\eta+\delta)}h + \Big(1+\frac{2\beta}{\alpha|\log (\delta/\eta)|} \Big) F_{\eps_k}\big(Ax + u_k,\,  B^d_\rho(x') \big)\\ &&\mkern50mu +\frac\eta{2h}\Big(1+\frac{2\beta}{\alpha|\log (\delta/\eta)|} \Big)\Big( 2{h}c\rho^d+\frac{1}{\alpha}F_{\eps_k}\big(Ax + u_k,\,   B^d_\rho(x') \big)\Big).
\end{eqnarray*}
Fix now $r>1$, and note that for $k$ large enough we have $\tau_k + B^d_\rho(x') \subset B^d_{r \rho}(x'')  \Subset \omega$. 
Since $v_k \to 0$ in $L^p\big(B^d_{r \rho}(x') \times (-h,h), \mathbb{R}^m\big)$:

\begin{eqnarray*}
        F_0\big(Ax,  B^d_\rho(x'') \big)
        &\leq &\, F_0\big(Ax,B^d_{r \rho}(x'')\big)\\
        &\leq &\, \liminf_{k\to+\infty} F_{\eps_k}\big(Ax + v_k,B^d_{r \rho}(x'') \big)\\
        &\leq &\, \Big(1 + \frac{\eta}{\alpha}\Big)\Big(1+\frac{2\beta}{\alpha|\log (\delta/\eta)|} \Big) \liminf_{k\to+\infty} F_{\eps_k}\big(Ax + u_k, B^d_{\rho}(x') \big)\\
        && \mkern40mu+ \frac{c\rho^d\eta}h\Big(1+\frac{2\beta}{\alpha|\log (\delta/\eta)|} \Big)+\frac{\beta c\rho^d(\eta+\delta)}h\\
        && \mkern40mu + \frac\beta h\big(1 + \abs A \abs ^p\big) \big\abs B^d_{r \rho}(x'') \setminus B^d_{ \rho}(x'')\big\abs  \\
        &\leq &\, \Big(1 + \frac{\eta}{\alpha}\Big)\Big(1+\frac{2\beta}{\alpha|\log (\delta/\eta)|} \Big) F_0\big( u_k, B^d_{\rho}(x') \big)\\
        && \mkern40mu+ \frac{c\rho^d\eta}h\Big(1+\frac{2\beta}{\alpha|\log (\delta/\eta)|} \Big)+\frac{\beta c\rho^d}h(\eta+\delta)\\
        && \mkern40mu + \frac\beta h\big(1 + \abs A \abs ^p\big) \big\abs B^d_{r \rho}(x'') \setminus B^d_{ \rho}(x'')\big\abs . 
\end{eqnarray*}
Letting first $r \to 1$, then $\eta \to 0$ and eventually $\delta\to 0$, we finally obtain the inequality
\begin{equation*}
    F_0\big(Ax, B^d_{ \rho}(x'') \big) \leq F_0\big(Ax, B^d_{ \rho}(x') \big).
\end{equation*}
By symmetry we then obtain also equality and the claim.
\end{proof}

We can then proceed in the proof of Theorem \ref{thm:MainResult}. 
Take $\omega = (0,1)^d$ and let $\{\eps_{k}\}$ be a subsequence given by \eqref{eq:compek}. The $\Gamma$-limit of the family $F_{\eps_k}$ exists and, for Proposition \ref{proposition:f0hom}, is equal to
\begin{equation*}
    F_0(u) = \int_{(0,1)^d} f_0(\nabla u)\,\mathrm dx.
\end{equation*}
This functional is sequentially weakly lower semicontinuous on $\mathrm W^{1,p}(\Omega_1,\mathbb{R}^m)$ so the function $f_0$ is $\mathrm W^{1,p}$-quasiconvex. This means we can write
\begin{equation}\label{eq: f0quasiconvex}
    f_0(A) = \min \biggl\{ \int_{(0,1)^d} f_0(A+\nabla u)\,\mathrm dx: u = 0 \hbox{ on }\partial(0,1)^d\biggr\} 
\end{equation}
for all $A \in \mathbb{M}^{m\times d}$.
From the property of convergence of minima for $\Gamma$-convergence and a change of variable in the right-hand side we obtain
\begin{eqnarray}
     f_0(A) & =&\nonumber \min \bigg\{ \int_{(0,1)^d} f_0(A+\nabla u)\,\mathrm dx:u = 0 \hbox{ on }\partial(0,1)^d \bigg\}\\
     & =&\nonumber  \min\bigg\{ \frac{1}{2h}\int_{(0,1)^d \times (-h,h)} f_0(A+\nabla_xu)\,\mathrm dx:u = 0 \hbox{ on } \partial\big((0,1)^d\big) \times (-h,h) \bigg\}\\
     & = &\nonumber \lim_{k\to+\infty} \inf\biggl\{ \frac{1}{2h} \int_{(0,1)^d\times(-h,h)}f\left(\frac{x}{\eps_k},y,\Big(A+\nabla_x u\Big|\frac{1}{\eps_k} \partial_y u\Big)\right) \,\mathrm dx\,\mathrm dy : {u \in \mathcal{W}_1}\biggr\}\\
     & = &\lim_{k\to+\infty} \inf\biggl\{ \frac{1}{2hT_k^d}\int_{(0,T_k)^d\times(-h,h)}f\big(x,y,(A+\nabla_x u|\partial_y u)\big) \,\mathrm dx\,\mathrm dy : {u \in \mathcal{W}_{T_k}}\biggr\},
\end{eqnarray}
where $T_k=1/\e_k$.

The following proposition states the existence of the limit in \eqref{eq:HomogenizationFormula}, giving in particular the independence of $f_0$ from the subsequence.

\begin{pro}\label{existence-formula} For every $A \in \mathbb{M}^{m \times d}$ the following limit exists
\begin{equation}
    f_{\hom}(A) = \lim_{T\to+\infty} \inf_{u \in \mathcal{W}_T}\biggl\{ \frac{1}{2hT^d}\int_{(0,T)^d\times(-h,h)}f\big(x,y,(A+\nabla_x u| \partial_y u)\big) \,\mathrm dx\,\mathrm dy \biggr\}.
\end{equation}
\end{pro}

\begin{proof}
For every $T>0$ let
\begin{equation*}
            g_A(T) \coloneqq \inf_{u \in \mathcal{W}_T} \biggl\{ \frac{1}{2hT^d}\int_{(0,T)^d\times(-h,h)}f\big(x,y,(A+\nabla_x u| \partial_y u)\big) \,\mathrm dx\,\mathrm dy\biggr\}.
\end{equation*}
Fix $\delta >0$ and let $\eta$ be such that $\delta>\eta>0$. Now fix $T>0$ and let $u_T \in \mathcal{W}_T$ be such that 
\begin{equation*}
    \frac{1}{2hT^d}\underbrace{\int_{(0,T)^d\times(-h,h)}f\big(x,y,(A+\nabla_x u_T| \partial_y u_T)\big) \,\mathrm dx\,\mathrm dy }_{\eqqcolon C_T}\leq g_A(T) + \frac{1}{T}. 
\end{equation*}
We want to estimate $g_A(S)$ in terms of $g_A(T)$ when $S\gg T$. For this purpose, we construct test functions $u_S \in \mathcal{W}_S$ by a patchwork procedure, and we then exploit the almost-periodicity of the energy density $f$, extending $u_T$ to a function which has an energy contribution of $\mathrm o(\frac{1}{\eta})C_T$. 
More precisely, we apply Lemma \ref{Lemma:EstimateOnSlices} to the integrable functions defined for $t\in [0,h]$ by
\begin{equation*}
    g(t)\coloneqq \int_{(0,T)^d} |A+\nabla_x u_T(x,\pm  {t}),\partial_y  u_T(x,\pm {t})|^p\,\mathrm dx.
\end{equation*}
We then obtain two values $y_{\eta}^+ \in (h-\delta,h)$ and $y_{\eta}^- \in (h,-h+\delta)$ such that
\begin{eqnarray}\label{eq:SliceEstimate1} \nonumber
     && \int_{[0,T]^d \times {(y_{\eta}^+,h+\eta)}}f\big(x,y,\big( A+\nabla_x u_T(x,y_{\eta}^+)\big|\,\partial_yu_T(x,y_{\eta}^+) \big)\big)\,\mathrm dx\,\mathrm dy \\ \nonumber
     & \leq& \beta\int_{[0,T]^d \times {(y_{\eta}^+,h+\eta)}}\Bigl( 1 + \big|\big(A+\nabla_x u_T(x,y_{\eta}^+)\big|\,\partial_yu_T(x,y_{\eta}^+)\big)\big|^p\Bigr)\,\mathrm dx\,\mathrm dy \\ \nonumber
     & \leq& \beta T^d {|h+\eta-y_\eta^+|}+ \frac{\beta}{|\log (\delta/\eta)|} \int_{[0,T]^d \times (-h,h)}  \big|\big(A+\nabla_x u_T(x,y_{\eta}^+)\big|\,\partial_yu_T(x,y_{\eta}^+)\big)\big|^p\,\mathrm dx\,\mathrm dy \\
     &\leq& \beta\left(T^d(\eta+\delta) + \frac{C_T}{\alpha|\log (\delta/\eta)|}\right) 
\end{eqnarray}
and
\begin{eqnarray}\label{eq:SliceEstimate2}\nonumber
   &&\hskip-2cm\int_{[0,T]^d \times {(-h-\eta,y_{\eta}^-)}}f\Big(x,y,\big( A+\nabla_x u_T(x,y_{\eta}^+)\big|\,\partial_yu_T(x,y_{\eta}^+) \big)\Big)\,\mathrm dx\,\mathrm dy \\
   &&\leq \beta\left(T^d(\eta+\delta) + \frac{C_T}{\alpha|\log (\delta/\eta)|}\right).
\end{eqnarray}
Thus, we can modify $u_T$ setting
\[\widetilde{u}_T(x,y)\coloneqq \begin{cases} u_T(x,y_\eta^+) &\text{ if } y\ge y^+_\eta\\
u_T(x,y) &\text{ if } y_{\eta}^-\le y \le y_\eta^+ \\
u_T(x,y_\eta^-) &\text{ if } y\le y^-_\eta
\end{cases}\]
Let $L_\eta$ be the inclusion length related to $\mathcal{T}_{\eta}$ and let $S > T + L_\eta$. Define 
\[\mathcal{I}_S = \mathbb{Z}^d\cap \Big[0,\Big\lfloor\frac{S}{T+L_\eta}\Big\rfloor-1\Big)^d\]
and, for every $\ell \in \mathcal{I}_S$, choose 
\begin{equation*}
    \tau_\ell\in \big((T+L_\eta)\ell + [0,L_\eta]^d)\big) \cap \mathcal T_\eta,
\end{equation*}
and the related $z_\ell$.
We then define a new test function by
\begin{equation*}
    u_S(x,y) = 
    \begin{cases}
    \widetilde{u}_T(x-\tau_\ell,y- z_\ell) &(x,y) \in \big(\tau_\ell + (0,T)^d\big)\times(-h,h)\\
    0, & \text{otherwise},
    \end{cases}
\end{equation*}
for every $(x,y)\in (0,S)^d\times(-h,h)$.

Note that $u_S$ is equal to zero on the set
\begin{equation*}
    Q_S \coloneqq \Big((0,S)^d \setminus \bigcup_{\ell \in \mathcal{I}_S}\tau_\ell + (0,T)^d\Big) \times(-h,h).
\end{equation*}
We have
\begin{equation*}
    |Q_S| \leq 2hS^d\bigg(1-\Big(\frac{T}{T+ L_ \eta} - \frac{T}{S}\Big)^d\bigg).
\end{equation*}
We can now estimate $g_A(S)$ using (\ref{eq:SliceEstimate1}) and (\ref{eq:SliceEstimate2}) 
\begin{equation*}
  \begin{split}
           2S^d g_A(S)  &\leq \int_{(0,S)^d \times (-h,h)}f\big(x,y, A+\nabla_xu_S(x,y),\partial_yu_S(x,y)\big)\,\mathrm dx\,\mathrm dy\\
          & = \sum_{\ell\in \mathcal{I}_S} \int_{(\tau_\ell + (0,T)^d)\times (-h,h)}f\big(x,y,A+\nabla_x\widetilde{u}_T(x-\tau_\ell,y- z_\ell),\partial_y\widetilde{u}_T(x-\tau_\ell,y- z_\ell)\big)\,\mathrm dx\,\mathrm dy \\ &\mkern90mu + \int_{Q_S}f(x,A)\,\mathrm dx\,\mathrm dy\\
          & \leq \sum_{\ell\in \mathcal{I}_S} \int_{ (0,T)^d\times (-h,h)}f\big(\tau_\ell+x,z_\ell+y, A+\nabla_x {u}_T(x,y),\partial_y{u}_T(x,y)\big)\,\mathrm dx\,\mathrm dy\\
          &\mkern90mu +\Big\lfloor\frac{S}{T+L_\eta}\Big\rfloor^d \frac{2 \beta C_T}{\alpha|\log(\delta/\eta)|} + \Big\lfloor\frac{S}{T+L_\eta}\Big\rfloor^d 2T^d\beta(\delta+\eta) + \int_{Q_S}f(x,A)\,\mathrm dx\,\mathrm dy
          \end{split}
\end{equation*}

Now, exploiting the almost-periodicity of $f$ we get
\begin{equation*}
    \begin{split}
         2S^d g_A(S)  &\leq \Big\lfloor\frac{S}{T+L_\eta}\Big\rfloor^d \bigg( \int_{ (0,T)^d\times (-h,h)}f\Big(x,y, \big(A+\nabla_x {u}_T(x,y)\big|\partial_y{u}_T(x,y)\Big)\,\mathrm dx\,\mathrm dy\\
          & \mkern90mu + \eta \int_{ (0,T)^d\times (-h,h)} 1 + \Big|\big(A+\nabla_xu_T(x,y)\big|\partial_yu_T(x,y)\big)\Big|^p\,\mathrm dx\,\mathrm dy\bigg)\\
          &\mkern90mu +\Big\lfloor\frac{S}{T+L_\eta}\Big\rfloor^d \frac{2 \beta C_T}{\alpha|\log(\delta/\eta)|} + \Big\lfloor\frac{S}{T+L_\eta}\Big\rfloor^d 2T^d\beta(\delta+\eta) + \int_{Q_S}f(x,A)\,\mathrm dx\,\mathrm dy
    \end{split}
\end{equation*}
Using the p-growth conditions
\begin{equation*}
    \begin{split}
         2S^d g_A(S)  &\leq \Big(1+\frac{\eta}{\alpha}\Big)\Big\lfloor\frac{S}{T+L_\eta}\Big\rfloor^d{ \int_{ (0,T)^d\times (-h,h)}f\Big(x,y, \big(A+\nabla_x {u}_T(x,y)\big|\partial_y{u}_T(x,y)\Big)\,\mathrm dx\,\mathrm dy}\\
          &+2\eta h T^d\Big\lfloor\frac{S}{T+L_\eta}\Big\rfloor^d  +\Big\lfloor\frac{S}{T+L_\eta}\Big\rfloor^d \frac{2\beta C_T}{\alpha |\log(\delta/\eta)|} + \Big\lfloor\frac{S}{T+L_\eta}\Big\rfloor^d 2T^d\beta(\delta+\eta)\\
          & \mkern90mu  + 2\beta h S^d\bigg(1-\Big(\frac{T}{T+ L_\eta} - \frac{T}{S}\Big)^d\bigg) \big(1+|A|\big)^p
    \end{split}
\end{equation*}
We now exploit   $C_T/(2T^d)\le g_A(T)+1/T$
\begin{equation*}
    \begin{split}
         g_A(S) &\leq \Big(\frac{T}{T+L_\eta}\Big)^d\Big(1+\frac{\eta}{\alpha}\Big)\Big(1+\frac{2\beta}{\alpha|\log (\delta/\eta)|}\Big)\Big(g_A(T)+\frac{1}{T}\Big)\\
          & +\Big(\eta h + \beta(\delta+\eta)\Big)\Big(\frac{T}{T+L_\eta}\Big)^d + \beta h \bigg(1-\Big(\frac{T}{T+ L_\eta} - \frac{T}{S}\Big)^d\bigg) \big(1+|A|\big)^p
    \end{split}
\end{equation*}
Taking the limit, first as $S \to \infty$ and then as $T \to \infty$
\begin{equation*}
         \limsup_{S\to+\infty} g_A(S) \leq \Big(1+\frac{\eta}{\alpha}\Big)\Big(1+\frac{2\beta}{\alpha|\log (\delta/\eta)|}\Big)\liminf_{T\to+\infty} g_A(T)+ \eta h + \beta(\delta+\eta)
\end{equation*}
We first send $\eta \to 0$ and obtain
\begin{equation*}
         \limsup_{S\to+\infty} g_A(S) \leq \liminf_{T\to+\infty}  g_A(T) + \beta\delta,
\end{equation*}
and then conclude by the arbitrariness of $\delta$.
\end{proof}

Eventually, it suffices to remark that Proposition \ref{existence-formula}  implies that $f_0(A)$ is independent on the subsequence $\{\eps_k\}$,  so that we simultaneously have the existence of the limit as $\e\to0$ and the representation \eqref{eq:ffhom}, which concludes the proof of the theorem.

\bigskip
{\bf Acknowledgments.} Andrea Braides is a member of GNAMPA, INdAM.

\end{document}